\documentclass[12pt,reqno]{amsart}
\usepackage{amsmath,amssymb,physics}
\usepackage{amsthm,amsopn,amssymb,a4wide}
\usepackage{enumitem}
\usepackage[mathscr]{eucal}
\usepackage[colorlinks]{hyperref}
\usepackage[capitalize,nameinlink]{cleveref}
\hypersetup{urlcolor=blue, citecolor=red, linkcolor=blue}

\usepackage{color, bm, amscd, tikz-cd}
\newcommand{\ran}{\operatorname{ran}}

\newcommand{\spncl}{\overline{\operatorname{span}}}
\newcommand{\supp}{\operatorname{supp}}

\newcommand{\A}{\mathcal A}
\newcommand{\C}{\mathbb{C}}
\newcommand{\Z}{\mathbb{Z}}
\newcommand{\R}{\mathbb{R}}
\newcommand{\N}{\mathbb{N}}
\newcommand{\D}{{\mathbb D}}

\newcommand{\T}{\mathbb{T}}
\newcommand{\B}{\mathscr{B}}
\renewcommand{\S}{\mathscr S}
\newcommand{\Sc}{\mathcal S}
\newcommand{\M}{\mathcal M}

\renewcommand{\H}{\mathcal{H}}
\newcommand{\K}{\mathcal{K}}

\newcommand{\G}{\mathcal G}
\newcommand{\F}{\mathcal F}

\newcommand{\la}{ \langle }
\newcommand{\ra}{\rangle}

\newcommand{\overbar}[1]{\mkern 1.5mu\overline{\mkern-1.5mu#1\mkern-1.5mu}\mkern 1.5mu}
\newcommand{\ob}{\overbar}
\newcommand{\wt}{\widetilde}

\newtheorem{thm}{Theorem}[section]

\newtheorem{corollary}[thm]{Corollary}
\newtheorem{lemma}[thm]{Lemma}

\newtheorem{definition}[thm]{Definition}
\newtheorem{remark}[thm]{Remark}

\newtheorem{example}[thm]{Example}

\numberwithin{equation}{section}

\def\textmatrix#1&#2\\#3&#4\\{\bigl({#1 \atop #3}\ {#2 \atop #4}\bigr)}
\def\dispmatrix#1&#2\\#3&#4\\{\left({#1 \atop #3}\ {#2 \atop #4}\right)}
\numberwithin{equation}{section}

\def\textmatrix#1&#2\\#3&#4\\{\bigl({#1 \atop #3}\ {#2 \atop #4}\bigr)}
\def\dispmatrix#1&#2\\#3&#4\\{\left({#1 \atop #3}\ {#2 \atop #4}\right)}

\begin{document}
	
	\title[]{Doubly commuting and dual doubly commuting semigroups of isometries}

	\author{Tirthankar Bhattacharyya, Shubham Rastogi and Vijaya Kumar U.}
\address{Department of Mathematics\\
	Indian Institute of Science\\
	Bangalore 560012, India}
\email{tirtha@iisc.ac.in; shubhamr@iisc.ac.in; vijayak@iisc.ac.in}
	
	\maketitle
	
	\renewcommand{\thefootnote}{\fnsymbol{footnote}}
	
	\footnotetext{MSC: Primary: 47D03, 47A65. Secondary: 47B91.\\
		Keywords: The shift semigroup, Isometric semigroups, Doubly commuting, Dual doubly commuting.}
	
	\begin{abstract}
		Structures of commuting semigroups of isometries under certain additional assumptions like double commutativity or dual double commutativity are found.
	\end{abstract}
\section{Introduction}
     One of the building blocks of  $C_0$-semigroups of isometries is the {\em right-shift-semigroup} $\S^\F=(\S^\F_t)_{t\ge 0}$ on $L^2(\R^+,\F)$ for any Hilbert space  $\F.$ It is defined as
     \[(\S_t^\F f)x=\begin{cases}
     	f(x-t) &\text{if } x\ge t,\\
     	0 & \text{otherwise,}
     \end{cases}\]
     for $f\in L^2(\R_+,\F)$ and it has been shown by Cooper in \cite{Cooper} to be one of the direct summands in the structure theorem for any $C_0$-semigroup of isometries. We shall denote $\S_t^\C$ by just $\S_t.$ Sometimes we shall use the identification of  $L^2(\R_+,\F)$ with $L^2(\R_+)\otimes \F$ which transforms $\S_t^\F$ into $\S_t\otimes I_\F.$

     All semigroups in this note are $C_0$-semigroups (i.e., strongly continuous one parameter semigroups). Two such semigroups $V_1=(V_{1,t})_{t\ge 0}$ and $V_2=(V_{2,s})_{s\ge 0}$   on a Hilbert space $\H$  are said to be

     \begin{enumerate}
     	\item {\em commuting} if
     	\begin{equation*}
     		V_{1,t}V_{2,s}=V_{2,s}V_{1,t} \text{ for all } t,s\ge 0.
     	\end{equation*}
     \item {\em doubly commuting} if they are commuting and
     \begin{equation*}
     	V_{1,t}V_{2,s}^*=V_{2,s}^*V_{1,t} \text{ for all } t,s\ge 0.
     \end{equation*}
   \end{enumerate}
We completely characterize pairs of doubly commuting semigroups of isometries in terms of concrete models in \cref{thm:DCcnumodel} and \cref{thm:cnuUmodel}. See \cite{Binzar-Lazureanu 2017} and \cite{Binzar-Lazureanu 2018} for a preliminary result of Cooper type decomposition obtained by Bînzar and  Lăzureanu.

This can be done in two ways. To begin with, we use the identification of the right-shift-semigroup with a semigroup of certain  multiplication operators on a vector valued Hardy space given in \cite[Theorem 3.4]{BCL}. This gives rise to certain partial isometries. We find commutants of these partial isometries. This gives a proof of \cref{thm:DCcnumodel}. On the other hand, the result \cite[Lemma 12.1]{BCL} of Berger-Coburn-Lebow is also useful for proving  \cref{thm:DCcnumodel}.

The minimal unitary extension of a pair of commuting semigroups of isometries is known due to \cite{Slocinski-dilation,Ptak}. Therefore, the dual can be defined for a pair of commuting semigroups of isometries. We obtain a Cooper type decomposition for a pair of dual doubly commuting semigroups of isometries in \cref{dualdoublecommuting} and a model in \cref{dual-model}.

We state below Cooper's theorem. The motivation comes from the celebrated Wold decomposition in \cite{Wold}. A semigroup $V=(V_t)_{t\ge 0}$ of isometries on a Hilbert space $\H$ is said to be {\em completely nonunitary} (c.n.u) if there is no reducing subspace $\H_0$ of $\H$ such that $V|_{\H_0}=(V_t|_{\H_0})_{t\ge 0}$ is a unitary semigroup.

 \begin{thm}[Cooper \cite{Cooper}]\label{Cooper}
 	Let $V=(V_t)_{t\ge 0}$ be a strongly continuous semigroup of isometries on a Hilbert space $\H.$ Then  $\H$ uniquely decomposes as the direct sum
 	\begin{equation*}\label{CooperDecomp}
 	\H=\H_p\oplus \H_u
 	\end{equation*}
 	of reducing subspaces $\H_p$ and $\H_u$ of $V,$ so that
 	\begin{enumerate}
 		\item $V|_{\H_p}$ is a c.n.u semigroup of isometries, and
 		\item $V|_{\H_u}$ is a  unitary semigroup.
 	\end{enumerate}
 Moreover, any c.n.u semigroup of isometries is unitarily equivalent to the right-shift-semigroup $\S^\F=(\S_t^\F)_{t\ge 0}$ on $L^2(\R_+,\F)$ for some Hilbert space $\F.$
 	\end{thm}

 An astute reader would notice that this paper is motivated by structure theorems for commuting isometries obtained over decades in \cite{BKPS}, \cite{GG}, \cite{Popovici II} and \cite{Popovici} to quote a few.

\section{On the right-shift-semigroup}

For any Hilbert space $\F,$ let   $H^2(\D,\F)$ denote the  Hardy space of $\F$-valued analytic functions over the unit disc $\D.$
We just write $H^2(\D)$ instead of $H^2(\D,\C).$ The space $H^2(\D,\F)$ is naturally identified with  $H^2(\D)\otimes\F.$

It is well known that (see \cite{Nagy-Foias}) the right-shift-semigroup $\S^\F$ on $L^2(\R_+,\F)$ is unitarily equivalent to the semigroup $(M_{\psi_t})_{t\ge 0}$ of multiplication operators on the vector valued Hardy space  $H^2(\D,\F),$ where $\psi_t:\D\to \B(\F)$ are the operator valued bounded analytic functions defined by $$\psi_t(z)=e^{t\frac{z+1}{z-1}} I_\F$$ for $z\in \D$ and $t\ge 0.$

On the other hand, in \cite[Theorem 3.4]{BCL}, Berger, Coburn and Lebow obtained another identification of the right-shift-semigroup $\S^\F$ with a semigroup of multiplication operators on the vector valued Hardy space $H^2(\D)\otimes L^2(\T, \F)$ where $\T$ is the unit circle and $\F$ is a separable Hilbert space. 
In this section, we recall their identification of $\S^\F$   rephrased on the vector valued Hardy space $H^2(\D,L^2([0,1], \F)).$ The subsequent lemma is of independent interest and is also useful to give different proofs of \cref{thm:DCcnumodel} and  \cref{normal-cont-comm}  for the case of separable Hilbert spaces.
\begin{thm}\label{thm:shiftsemigp}
	The right-shift-semigroup $\S^\F=(\S_t^\F)_{t\ge 0}$ on $L^2(\R_+,\F)$ is unitarily equivalent to the semigroup $(M_{\varphi_t^\F})_{t\geq0}$ on $H^2(\D, L^2([0,1],\F)),$ where for $n\in \Z_+$ and $n\leq t< n+1,$ the multiplier $\varphi^\F_t:\D\to \B(L^2([0,1],\F))$ is given by
	\begin{equation}\label{phi_t}
		\varphi^\F_t(z)=E^\F_{0,t-n}z^n+E^\F_{1,t-n}z^{n+1}
	\end{equation}
	and for $0\leq s<1,$ $E^\F_{0,s},E^\F_{1,s}$ are the partial isometries in $\B(L^2([0,1],\F))$ given by
	\begin{equation}
		(E^\F_{0,s}f)(x)=\begin{cases}
			0 & \text{if } x<s,\\
			f(x-s) &\text{if }s\le x\le 1,
		\end{cases}
	\end{equation} and
	\begin{equation}
		(E^\F_{1,s}f)(x)=\begin{cases}
			f(1-s+x) & \text{if } x\le s,\\
			0 &\text{if }s< x\le 1.
		\end{cases}
	\end{equation}
\end{thm}
\begin{proof}
	Consider the unitary $W:L^2(\R_+,\F)\to H^2(\D,L^2([0,1],\F))$ defined by
	\begin{equation}\label{eq:unitaryW}
		Wf=\sum_{n=0}^{\infty}f_nz^n
	\end{equation}
	where $f_n(\alpha):=f(n+\alpha)$ for $\alpha\in [0,1]$ and $n\in \Z_+.$ Let $0\leq t<1$ and $g\in L^2([0,1],\F).$ $$W\S_t^\F W^*(gz^n)=W\S_t^\F(h)= (E^\F_{0,t}g)z^n+(E^\F_{1,t}g)z^{n+1}=M_{\varphi^\F_t}(gz^n),$$ where $h\in L^2(\R_+,\F)$ is given by
	\begin{equation*}
		h(\alpha)=\begin {cases}
		g(\alpha-n) & \text{if } n\leq \alpha\leq n+1,\\
		0 & \text{else.}
	\end{cases}
\end{equation*}
It is evident that $W\S_1^\F W^*=M_{\varphi_1^\F}=M_z^{L^2([0,1],\F)},$ the multiplication operator by $z$ on $H^2(\D,L^2([0,1],\F)).$
Hence $W\S_t^\F W^*=M_{\varphi^\F_t}$ for $0\leq t\leq1.$ For $n\in \N$ and $n\leq t<n+1,$ since $\S_t^\F=(\S_1^\F)^n\S_{t-n}^\F,$ we get
\begin{align*}
	W\S_t^\F W^*&=W(\S_1^\F)^nW^*W\S_{t-n}^\F W^*\\
	&=(W\S_1^\F W^*)^nW\S_{t-n}^\F W^*\\
	&=(M_{z}^{L^2([0,1],\F)})^nM_{\varphi^\F_{t-n}}\\
	&=M_{z^n\varphi^\F_{t-n}}\\
	&=M_{\varphi^\F_t}.
\end{align*}This completes the proof.
\end{proof}
\begin{lemma}\label{CommutingWithE0sandE1s}
	Let $\F$ be a separable Hilbert space and  $\Lambda: L^2[0,1]\otimes\F \to L^{2}([0,1],\F)$  be the natural unitary isomorphism. 	Suppose $B\in \B(L^{2}([0,1],\F))$ commutes with $E^\F_{0,s}$ and $E^\F_{1,s}$ for all $0\leq s<1.$ Then, $B=\Lambda(I_{L^{2}[0,1]}\otimes C)\Lambda^*$ for some $C\in \B(\F).$
\end{lemma}

\begin{proof}
For any $0\le s< 1,$ note that $$\ran E^\F_{0,s}=\{f: \supp (f)\subseteq[s,1]\} \ and \  \ran E^\F_{1,s}=\{f: \supp (f)\subseteq[0,s]\}.$$ Hence $L^2([0,1],\F)=\ran E^\F_{0,s}\oplus \ran E^\F_{1,s}$ for any $0\le s<1.$ Since $B$ commutes with $E^\F_{0,s}$ and $E^\F_{1,s}$ for all $0\leq s<1,$  the ranges of $E^\F_{0,s}$ and the ranges of $E^\F_{1,s}$ are reduced by $B.$ For any $x\in \F$ and an interval $I\subset [0,1],$ let $\mathbf{1}_I\cdot x:[0,1]\to \F$ denotes the function  $$(\mathbf{1}_I\cdot x)(\alpha)=\begin{cases}
	x & \text{if } \alpha \in I,\\
	0 & \text{else.}
\end{cases}$$  Let  $f_x:=B(\mathbf{1}_{[0,1]}\cdot x).$ Note that
$$B(\mathbf{1}_{[0,1]}\cdot x)=B(\mathbf{1}_{[0,c]}\cdot x)+B(\mathbf{1}_{(c,d)}\cdot x)+B(\mathbf{1}_{[d,1]}\cdot x)=f_x$$
As $B$ reduces the ranges of $E^\F_{1,c}$ and $E^\F_{0,d},$ we have
\[B(\mathbf{1}_{(c,d)}\cdot x)(\alpha)=\begin{cases}
	f_x(\alpha) & \text{if } \alpha \in (c,d),\\
	0 & \text{else.}
\end{cases}\]
Hence $B(\mathbf{1}_{(c,d)}\cdot x)=\mathbf{1}_{(c,d)}f_x$ for all $0\le c<d\le 1.$
With a little more work, $f_x$ can be shown to be a bounded measurable function (and not just in $L^2([0,1],\F)$), but we do not need this fact.

Since $B$ commutes with $E^\F_{0,s}$ and $E^\F_{1,s}$ for $0\leq s<1,$ we have
\[B( \mathbf{1}_{[s,1]}\cdot x)=E^\F_{0,s}(f_x)\quad \text{and}\quad B(\mathbf{1}_{[0,s]}\cdot x)=E^\F_{1,s}(f_x).\]
Therefore,
\begin{equation}\label{translation}
	B(\mathbf{1}_{[0,1]}\cdot x)=E^\F_{0,s}(f_x)+E^\F_{1,s}(f_x)=f_x\text{ for } 0\le s<1.
\end{equation}
Extend $f_x$ on $\R$ periodically with period $1,$ say $\tilde{f_x}.$ Note that the translation  of $\tilde{f_x}$ by any $0\leq s<1$ equals $\tilde{f_x}$ almost everywhere by \cref{translation}. This implies that $f_x=\mathbf{1}_{[0,1]}\cdot y$ for some $y\in \F.$ See also \cite[Problem 4, Chap. 7]{Rudin}.

Let $\Theta:\F\to L^2([0,1],\F)$ be the isometric embedding $\Theta x=\mathbf 1_{[0,1]}\cdot x.$ Use $\Theta $ to define $C$ by $C=\Theta^* B\Theta.$
Let $\Lambda: L^2[0,1]\otimes\F \to L^{2}([0,1],\F)$  be the natural unitary. Then
\begin{align*}
	B(\mathbf 1_{(c,d)}\cdot x)&=\mathbf 1_{(c,d)}B(\Theta(x))\\&=\mathbf 1_{(c,d)}\cdot \Theta^*B\Theta(x)\\
	&=\mathbf 1_{(c,d)}\cdot Cx\\
	&=\Lambda(I_{L^2[0,1]}\otimes C)\Lambda^*(\mathbf 1_{(c,d)}\cdot x).
\end{align*}
Since the characteristic functions $\mathbf{1}_{(c,d)}$ for  $0\le c<d\le 1$ are total in $L^2[0,1],$ we get $B=\Lambda(I_{L^{2}[0,1]}\otimes C)\Lambda^*.$  This completes the proof.
\end{proof}

\section{Doubly commuting  semigroups of isometries}\label{Doubly}

\begin{example}\label{DCexample}
	For $j=1,2$ define $\Sc_{j,t}:L^2(\R_+^2)\to L^2(\R_+^2)$ by
	\begin{align*}
		(\Sc_{1,t}f)(x_1,x_2)=\begin{cases}
			f(x_1-t,x_2) & \text{if }x_1\ge t\\
			0 & \text{else}
		\end{cases}, (\Sc_{2,t}f)(x_1,x_2)=\begin{cases}
			f(x_1,x_2-t) & \text{if }x_2\ge t\\
			0 & \text{else.}
		\end{cases}
	\end{align*}
Then  $\Sc_1:=(\Sc_{1,t})_{t\ge 0}$ and $\Sc_2:=(\Sc_{2,t})_{t\ge 0}$ are doubly commuting semigroups of isometries on $L^2(\R_+^2).$
\end{example}
 A pair of semigroups $(V_1,V_2)$ on $\H$ is said to be {\em jointly unitarily equivalent} to $(W_1,W_2)$ on $\K$ if there is a unitary $U:\H\to \K$ such that
 \begin{equation*}
 	W_{j,t}=UV_{j,t}U^*\text{ for all }t\ge 0, j=1,2.
 \end{equation*}
In the natural isomorphism of  $ L^2(\R_+^2)$ and $L^2(\R_+)\otimes L^2(\R_+)$, we can see that the pair of semigroups of isometries $(\Sc_1,\Sc_2)$  is jointly unitarily equivalent to the pair of semigroups of isometries $(\S\otimes I_{L^2(\R_+)},I_{L^2(\R_+)}\otimes \S),$ where $\S$ is the right-shift-semigroup on $L^2(\R_+).$

\begin{definition}
	A pair $(V_1,V_2)$ of semigroups is called a {\em bishift-semigroup} if $(V_1,V_2)$ is jointly unitarily equivalent to $(\Sc_1\otimes I_\F,\Sc_2\otimes I_\F)$ on $L^2(\R_+^2)\otimes \F$ for some Hilbert space $\F.$
\end{definition}

The following theorem is known due to Bînzar and Lăzureanu \cite{Binzar-Lazureanu 2017,Binzar-Lazureanu 2018}. 
\begin{thm}[\cite{Binzar-Lazureanu 2018}]\label{doublecommuting}
	Let $V_1$ and $V_2$ be doubly commuting semigroups of isometries on $\H,$ then there exists a decomposition
	$$\H=\H_{p,p}\oplus\H_{p,u}\oplus\H_{u,p}\oplus\H_{u,u}$$
	where $\H_{i,j}$ reduces both $V_1$ and $V_2$ for all $i,j\in\{p,u\}$ such that \\
	(1) both $V_1$ and $V_2$ are c.n.u semigroups  on $\H_{p,p},$\\
	(2) $V_1$ is a c.n.u semigroup, $V_2$ is a  unitary semigroup on $\H_{p,u},$\\
	(3) $V_1$ is a  unitary semigroup, $V_2$ is a c.n.u semigroup on $\H_{u,p}$  and\\
	(4) both $V_1$ and $V_2$ are  unitary semigroups on  $\H_{u,u}.$
\end{thm}

In the rest of this section we shall give models for $\H_{p,p},\H_{p,u}$ and $\H_{u,p}$ parts in the theorem above. The following lemma describes the operators doubly commuting with the shift. Let $M_z$ denote the multiplication by $z$ on the Hardy space of the unit disc $H^2(\D).$ 
\begin{lemma}\label{Double-Comm-Mz}
	If $N$ is a bounded operator on $H^2(\D)\otimes \F$ doubly commuting with $M_z\otimes I_\F,$ then $N=I_{H^2(\D)}\otimes \omega$ for some bounded  operator $\omega$ on $\F.$ In particular, any normal operator $N$ commuting with $M_z\otimes I_\F,$ is of the form $N=I_{H^2(\D)}\otimes \omega$ for some normal operator $\omega$ on $\F.$
\end{lemma}

\begin{proof}
	A proof of the first part of the lemma can be found in \cite[page 38]{Fil}. The part involving normality is then straightforward from Fuglede's theorem (\cite{Fuglede}).
\end{proof}
\begin{thm}\label{thm:DCcnumodel}
	Let $V_1$ and $V_2$ be two semigroups on $\H$. Then $V_1$ and $V_2$ are doubly commuting c.n.u semigroups of isometries if and only if $(V_1,V_2)$ is a bishift-semigroup. 
\end{thm}

\begin{proof}
	Let $\A=C^*\{V_{2,s}:s\ge 0\}\subseteq \B(\H)$ be the $C^*$-algebra generated by $\{V_{2,s}:s\ge 0\}.$ Suppose $V_1$ and $V_2$ are doubly commuting,  every element of $\A$ commutes with $V_{1,t}$ for all $t\ge 0.$  Therefore, by \cite[Lemma 12.1]{BCL} there is a Hilbert space $\G$ and a unitary $U:\H\to L^2(\R_+)\otimes \G$ such that
	$$UV_{1,t}U^*=\S_t\otimes I_\G\text{ and }\quad UV_{2,t}U^*=I_{L^2(\R_+)}\otimes B_t$$ for all $t\ge 0,$ where $B_t\in\B(\G).$ As $(V_{2,t})_{t\ge 0}$ is a c.n.u semigroup of isometries, we must have $(B_t)_{t\ge 0}$ to be a c.n.u semigroup of isometries on $\G.$ Therefore, by \cref{Cooper}, there is a Hilbert space $\F$ and a unitary $W:\G\to L^2(\R_+)\otimes \F$ such that
	$$WB_tW^*=\S_t\otimes I_\F\text{ for } t\ge 0.$$
	Let $Z: \H\to L^2(\R_+)\otimes L^2(\R_+)\otimes\F$ be the unitary $Z=(I\otimes W)U.$ Then
	\begin{align*}
		ZV_{1,t}Z^*&=(I_{L^2(\R_+)}\otimes W)(\S_t\otimes I_\G)(I_{L^2(\R_+)}\otimes W^*)=\S_t\otimes I_{L^2(\R_+)}\otimes I_\F \text{ and }\\
		ZV_{2,t}Z^*&=(I_{L^2(\R_+)}\otimes W)(I_{L^2(\R_+)}\otimes B_t)(I_{L^2(\R_+)}\otimes W^*)=I_{L^2(\R_+)}\otimes\S_t\otimes I_\F
	\end{align*}
	for all $t\ge 0.$ This shows that $(V_1,V_2)$ is a bishift-semigroup.
\end{proof}
We provide another proof of \cref{thm:DCcnumodel} using \cref{CommutingWithE0sandE1s} and \cref{Double-Comm-Mz}  when $\H$ is separable.
\begin{proof}[Another proof of \cref{thm:DCcnumodel}]
	Since $V_1$ is a c.n.u semigroup,  by \cref{Cooper} and \cref{thm:shiftsemigp}, there exists a separable Hilbert space $\F$ and a unitary $U:\H\to H^2(\D,L^2([0,1],\F))$ such that  $UV_{1,t}U^*=M_{\varphi^\F_t}$ where $\varphi^\F_t$ is  as given in \cref{phi_t}. Then $UV_{2,t}U^*$ doubly commutes with $M_{\varphi^\F_1}=M_z^{L^2([0,1],\F)}$ for all $t\ge 0.$ Therefore, by \cref{Double-Comm-Mz}, $UV_{2,t}U^*=M_{\psi_t}$ where  $\psi_t$ is the constant function given by $\psi_t(z)=B_t$ for all $z\in \D.$  This shows that $(V_1,V_2)$ is jointly unitarily equivalent to $((M_{\varphi^\F_t})_t,(M_{\psi_t})_t)$ on $H^2(\D,L^2([0,1],\F)).$
	
	Note that  $B=(B_t)_{t\ge 0}$ is a c.n.u semigroup of isometries on $L^2([0,1],\F).$ Since $\psi_t$ commutes with $\varphi^\F_s,$ $B_t$ commutes with $E^\F_{0,s}$ and $E^\F_{1,s}$ for all $t\ge 0$ and $0\le s<1.$  Therefore by \cref{CommutingWithE0sandE1s}, $\Lambda^* \psi_t(z)\Lambda =I_{L^2([0,1])}\otimes C_t$ for all $z\in \D,$ where $C=(C_t)$ is a c.n.u semigroup of isometries on $\F$ and $\Lambda: L^2[0,1]\otimes\F \to L^{2}([0,1],\F)$  is the natural unitary.   Note also that $\Lambda^*\varphi^\F_t(z)\Lambda=\varphi^\C_t(z)\otimes I_\F$ for all $z\in \D$ and $t\ge 0,$ since $\Lambda^* E^\F_{j,s}\Lambda=E^\C_{j,s}\otimes I_\F$ for $j=0,1$ and for all $0\le s< 1.$
	
	Therefore, $((M_{\varphi^\F_t})_t,(M_{\psi_t})_t)$ is jointly unitarily equivalent to $(\S\otimes I_\F, I_{L^2(\R_+)}\otimes C_t)$ by \cref{thm:shiftsemigp}, and this is jointly unitarily equivalent to $(\S\otimes I_{L^2(\R_+)}\otimes I_\K,I_{L^2(\R_+)}\otimes \S\otimes I_\K)$ for some Hilbert space $\K,$ by \cref{Cooper}, as $(C_t)$ is a c.n.u semigroup of isometries. This completes the proof.	
\end{proof}
We observe the following result which is analogous to \cref{Double-Comm-Mz}.  The proof follows from \cite[Lemma 12.1]{BCL} and Fuglede's theorem \cite{Fuglede}. 
We remark that an elementary proof can be given using the notion of cogenerators and \cref{Double-Comm-Mz} instead of invoking \cite[Lemma 12.1]{BCL}. We leave that to the reader.
\begin{thm}\label{normal-cont-comm}
	Let $A=(A_t)_{t\ge 0}$ be a  semigroup of bounded operators on $L^2(\R_+)\otimes \F.$  Then, $A$ doubly commutes with $\S\otimes I_\F$ if and only if there exists a semigroup $(B_t)_{t\ge 0}$ on $\F$ such that $A_t=I\otimes B_t$ for all $ t\ge 0.$
	In particular, a normal semigroup $A$ commutes with $\S\otimes I_\F$ if and only if there exists a normal semigroup $(B_t)_{t\ge 0}$ on $\F$ such that $A_t=I\otimes B_t$ for all $ t\ge 0.$
\end{thm}
As a corollary to the above theorem we get the model for the $\H_{p,u}$ (equivalently, for the $\H_{u,p}$) part.

\begin{corollary}\label{thm:cnuUmodel}
		Let $V_1$ be a c.n.u semigroup of isometries commuting with a unitary semigroup $V_2$ on $\H.$ Then, there is a Hilbert space $\F$ and a unitary isomorphism between $\H$ and $L^2(\R_+)\otimes \F$ so that under the unitary isomorphism, $(V_1,V_2)$ is jointly equivalent to $(\S\otimes I_\F, I_{L^2(\R_+)}\otimes U)$ for some unitary semigroup $U$ on $\F,$ where $\S$ is the right-shift-semigroup on $L^2(\R_+).$
\end{corollary}

\section{Dual doubly commuting semigroups of isometries}
\label{DDC}
Let $(V_1,V_2)$ be a pair of commuting semigroups of isometries on $\H,$ where $V_j=(V_{j,t})_{t\ge 0}$ for $j=1,2.$ Let $(\ob{V_1},\ob{V_2})$ on $\ob\H$ be the minimal unitary extension of the pair $(V_1,V_2),$ (this extension exists due to \cite{Slocinski-dilation, Ptak}). So, $\ob{V_1}=(\ob{V_{1,t}})_{t\in \R}$ and  $\ob{V_2}=(\ob{V_{2,s}})_{s\in \R}$ are commuting unitary groups on $\ob{\H}$ such that $V_{1,t}V_{2,s}=P_\H\ob{V_{1,t}}\ob{V_{2,s}}|_\H$ for $t,s\ge 0$ and $$\ob\H=\spncl\{\ob{V_{1,t}}\ob{V_{2,s}}(\H):s,t\in\R\}.$$ 
 It is clear that $\ob{\H}\ominus \H$ is an invariant subspace for $(\ob{V_{i,t}})^*$ for all $t\ge 0,i=1,2.$ Let $\wt{V_{i,t}}:={(\ob{V_{i,t}})}^*|_{\ob\H \ominus \H}$ for $i=1,2,t\ge 0.$ 
Then $(\wt{V_1},\wt{V_2})$ is a pair of commuting
semigroups of isometries on $\wt\H:=\ob\H\ominus\H,$ where $\wt{V_1}  =(\wt{V_{1,t}})_{t\ge 0}$ and $\wt{V_2}=(\wt{V_{2,s}})_{s\ge 0}.$ The pair $(\wt{V_1},\wt{V_2})$ is called the {\em dual} of $(V_1,V_2).$ The pair $(V_1,V_2)$ is said to be {\em dual doubly commuting} if the dual $(\wt{V_1},\wt{V_2})$ is doubly commuting. In the sequel, for a semigroup $W=(W_t)_{t\ge 0}$ the notation $W^*$ denotes the adjoint semigroup $(W_t^*)_{t\ge 0}.$

\begin{example}\label{DDC-ex}
	Let $\H=L^2(\R^2\setminus \R^2_+)$ and for $f\in\H,$ define
	\begin{align*}
		(\M_{1,t}f)(x_1,x_2)&=\begin{cases}
			0 & \text{if } x_2\ge 0 \text{ and }  x_1\ge -t,\\
			f(x_1+t,x_2) &  \text{else,}
		\end{cases}\\ (\M_{2,t}f)(x_1,x_2)&=\begin{cases}
			0 & \text{if }x_1\ge 0\text{ and }  x_2\ge -t,\\
			f(x_1,x_2+t)& \text{else.}
		\end{cases}
	\end{align*}
Then $(\M_1,\M_2)$ is a pair of commuting semigroups of isometries on $\H,$ where $\M_1=(\M_{1,t})_{t\ge 0}$ and $\M_2=(\M_{2,s})_{s\ge 0}.$ Let $\ob{\Sc_{j,t}}:L^2(\R^2)\to L^2(\R^2)$ for $j=1,2,t\in \R $ be defined by
\begin{align*}
	(\ob{\Sc_{1,t}}f)(x_1,x_2)=	f(x_1-t,x_2),\quad
	(\ob{\Sc_{2,t}}f)(x_1,x_2)=f(x_1,x_2-t) .
\end{align*}
for $f\in L^2(\R^2)$ and $(x_1,x_2)\in \R^2.$ Note that
	 $({(\ob{\Sc_1})}^*,{(\ob{\Sc_2})}^*)$ on $L^2(\R^2)$ is the minimal unitary extension of $(\M_1,\M_2).$ Therefore the dual of $(\M_1,\M_2)$ is $(\Sc_1,\Sc_2)$ on $L^2(\R_+^2).$ Hence $(\M_1,\M_2)$ is dual doubly commuting.
\end{example}

A pair $(V_1,V_2)$ of commuting semigroups of isometries on $\H$ is called {\em  completely nonunitary (c.n.u)} if there is no non-zero reducing subspace $\H_0$ of $\H$  for both $V_1$ and $V_2$ so that $V_i|_{\H_0}$ is a unitary semigroup for $i=1,2.$

\begin{remark}\label{rem}\leavevmode
	\begin{enumerate}
		\item \cref{Cooper} for a semigroup  $V=(V_t)_{t\ge 0}$ of isometries, in particular, gives us the Wold decomposition of $V_t$ for each $t> 0.$
		\item Let $(v_1,v_2)$ be a pair of commuting isometries. Let $\H_u$ be the unitary part in the Wold decomposition of  $v=v_1v_2.$  Then $\H_u$ is reducing for both $v_1$ and $v_2.$
		\item Using the above, it is easy to notice that $(V_1,V_2)$ is a c.n.u pair of semigroups of isometries if and only if the product semigroup $V=V_1V_2=(V_{1,t}V_{2,t})_{t\ge 0}$ is c.n.u.
		
	\end{enumerate}
\end{remark}
This section uses the ideas from \cite{Popovici II} extensively.

\begin{lemma}\label{dualc.n.u}
 The dual pair $(\wt{V_1},\wt{V_2})$ is always c.n.u.
\end{lemma}

\begin{proof}
	We have $\ob{V_{i,t}}=\begin{pmatrix}
		V_{i,t}&\star\\0&(\wt{V_{i,t}})^*
	\end{pmatrix}$ on $\ob\H=\H\oplus\wt\H$ for $t\ge 0, i=1,2.$ By  minimality, $\ob\H=\spncl\{\ob{V_{1,t}}\ob{V_{2,s}}(\H):t,s\in \R\}.$  
 Let $\wt V=\wt {V_1}\wt {V_2},$ that is, $\wt {V_t}=\wt {V_{1,t}}\wt {V_{2,t}}$  for all $t\ge 0.$  Let $\H_0$ be the unitary part in the Cooper's decomposition (\cref{Cooper}) of $\wt{V}.$ Suppose $\wt V$ is not c.n.u. Then $\H_0$ is a non-zero subspace of $\wt\H$ and $\H_0$  is reducing for $\wt{V_{i,t}}$ for all $t\ge 0,i=1,2$ (from \cref{rem}). Since $(\ob{V_{i,t}})^*|_{\H_0}=\wt{V_{i,t}}|_{\H_0},$  note that $\H_0$ is a reducing subspace for $\ob{V_{i,t}}$  for all $t\ge 0$ and $i=1,2.$ Therefore, $$\spncl\{\ob{V_{1,t}}\ob{V_{2,s}}(\H):t,s\in \R\}\perp \H_0.$$ This is a contradiction to the minimality.
\end{proof}

\begin{thm}\label{double-dual}
	Let $(V_1,V_2)$ be a c.n.u pair of commuting semigroups of isometries on $\H.$ If the pair  $(\ob{V_1},\ob{V_2})$ on $\ob{\H}$ is the minimal unitary extension of $(V_1,V_2),$ then the pair $((\ob{V_1})^*,(\ob{V_2})^*)$  is the minimal unitary extension of the dual  $(\wt{V_1},\wt{V_2}).$ In particular, $({\overset{\approx}{V_1}}, {\overset{\approx} {V_2}})=(V_1,V_2),$ where  $({\overset{\approx}{V_1}}, {\overset{\approx}{V_2}})$ is the dual of $(\wt{V_1},\wt{V_2}).$
\end{thm}

\begin{proof}
 Let  $\wt\H=\ob\H\ominus\H.$ Only minimality of the extension needs to be proved, i.e.,
\begin{equation*}\label{minimal}
 \spncl\{\ob{V_{1,t}}\ob{V_{2,s}}(\wt\H):t,s\in \R\}=\ob\H.
\end{equation*}
To that end, let $x\in \ob\H$ and  $x\perp  \spncl\{\ob{V_{1,t}}\ob{V_{2,s}}(\wt\H):t,s\in \R\}.$ Then $\la \ob{V_{1,t}}\ob{V_{2,s}}\tilde{h},x \ra = 0$ for all $\tilde{h}\in \wt\H $ and $t,s\in \R.$
This implies $\la \tilde{h}, \ob{V_{1,t}}\ob{V_{2,s}}x\ra=0$ for all $\tilde{h}\in \wt\H$ and $ t,s\in \R.$ Hence $\ob{V_{1,t}}\ob{V_{2,s}}x\in \H$ for all $ t,s\in \R.$ So, $X=\spncl\{\ob{V_{1,t}}\ob{V_{2,s}}x: t,s\in \R\}\subseteq \H.$ Clearly $X$ reduces both $\ob{V_1}$ and $\ob{V_2}$ to unitary semigroups. Since $X\subseteq \H$ we have $\ob{V_{i,t}}|_{X}=V_{i,t}|_{X}$ for $i=1,2$ and for all $t\ge 0.$ This shows that $X$ reduces both $V_1$ and $V_2$ to unitary semigroups. Now since $(V_1,V_2)$ is c.n.u,  $X$ has to be the zero space. In particular, $x=0.$ This shows the minimality. \end{proof}

\begin{definition}
		A c.n.u pair $(V_1,V_2)$ of commuting semigroups of isometries is called a {\em modified-bishift-semigroup} if the dual $(\wt{V_1},\wt{V_2})$ of $(V_1,V_2)$ is a bishift-semigroup.
\end{definition}
The following is the main result of this section. It gives a Cooper type decomposition for the pairs of dual doubly commuting semigroups of isometries.
\begin{thm}	\label{dualdoublecommuting}
	 Let  $(V_1,V_2)$ be  a pair of commuting semigroups of isometries on $\H$. Suppose $(V_1,V_2)$ is dual doubly commuting.  Then the space $\H$ decomposes as $$\H=\H_{m}\oplus \H_{p,u}\oplus\H_{u,p}\oplus\H_{u,u},$$ so that $\H_m,\H_{p,u},\H_{u,p}$ and $\H_{u,u}$ reduce both $V_1$ and $V_2,$ \\
	 (1) $(V_1,V_2)$ is a modified-bishift-semigroup on $\H_{m},$\\
	  (2) $V_1$ is a c.n.u semigroup, $V_2$ is a  unitary semigroup on $\H_{p,u},$\\
	  (3) $V_1$ is a  unitary semigroup, $V_2$ is a c.n.u semigroup on $\H_{u,p},$  and\\
	  (4) both $V_1$ and $V_2$ are  unitary semigroups on  $\H_{u,u}.$
\end{thm}

\begin{proof}
	First assume that $(V_1,V_2)$ is a c.n.u pair of dual doubly commuting semigroups of isometries. Let $(\ob{V_1},\ob{V_2})$ on $\ob{\H}$ be the minimal unitary extension of $(V_1,V_2)$  and let $(\wt{V_1},\wt{V_2})$ be the dual of $(V_1,V_2).$
	
	Since $(V_1,V_2)$ is dual doubly commuting, $(\wt{V_1},\wt{V_2})$ is doubly commuting. Hence by \cref{doublecommuting} and \cref{thm:DCcnumodel} we have $\wt\H=\wt\H_{p,p}\oplus\wt\H_{p,u}\oplus\wt\H_{u,p}\oplus\wt\H_{u,u},$ where $(\wt{V_1},\wt{V_2})$ is a bishift-semigroup on $\wt\H_{p,p},$   $\wt{V_1}$ is a c.n.u semigroup and $\wt{V_2}$ is a unitary semigroup on $\wt\H_{p,u},$ and $\wt{V_1}$ is a unitary semigroup and $\wt{V_2}$ is a c.n.u semigroup on $\wt\H_{u,p}.$ By \cref{dualc.n.u} we have $\wt\H_{u,u}=\{0\}.$

	For $i,j\in \{p,u\}$ and $(i,j)\ne (u,u),$
	let
	\[\hat\H_{i,j}=\spncl\{\ob{V_{1,t}}\ob{V_{2,s}}(\wt\H_{i,j}):t,s\in \R\}.\]
	Then, it is not difficult to show that  $\hat\H_{p,p},\hat\H_{p,u}$ and $\hat\H_{u,p}$ are pairwise orthogonal and $\ob\H=\hat\H_{p,p}\oplus\hat\H_{p,u}\oplus\hat\H_{u,p}.$ Also clearly $\hat\H_{i,j}$ is a reducing subspace for $\ob{V_{k,t}}$ for  $k=1,2,t\ge 0$ and $i,j\in\{p,u\},(i,j)\ne (u,u).$
	
	 Let $\H_m=\hat\H_{p,p}\ominus \wt\H_{p,p}$ and  $\H_{i,j}=\hat\H_{i,j}\ominus\wt\H_{i,j}$ for $i,j\in\{p,u\}$ with $ i\ne j.$  Then  \begin{equation}\label{H-decomp}
		\H=\H_m\oplus\H_{p,u}\oplus\H_{u,p}.
	\end{equation}

 Note that $\H_m,\H_{p,u}$ and $\H_{u,p}$ are invariant subspaces for $\ob{V_{k,t}}$ for all $t\ge 0, k=1,2.$ Hence they are invariant subspaces for $V_{k,t}.$ Thus by \cref{H-decomp}, they are reducing subspaces for $V_{k,t}.$

Note that the pair $(\ob{V_1}|_{\hat\H_{p,p}},\ob{V_2}|_{\hat\H_{p,p}})$  on $\hat\H_{p,p}$  is the minimal unitary extension for  $(V_1|_{\H_m},V_2|_{\H_m})$ on $\H_m$ and  the pair $(\ob{V_1}|_{\hat\H_{i,j}},\ob{V_2}|_{\hat\H_{i,j}})$  on $\hat\H_{i,j}$  is the minimal unitary extension for  $(V_1|_{\H_{i,j}},V_2|_{\H_{i,j}})$ on $\H_{i,j}$ for $i,j\in \{p,u\}$ and $i\ne j.$
	
	Now we shall show that $(V_1|_{\H_m},V_2|_{\H_m})$ is a modified-bishift-semigroup on $\H_m.$ To show that, we must show that $((\ob{V_1})^*|_{\wt\H_{p,p}},(\ob{V_2})^*|_{\wt\H_{p,p}})$ is a bishift-semigroup on $\wt\H_{p,p}.$ For this, note that $((\ob{V_1})^*|_{\wt\H_{p,p}},(\ob{V_2})^*|_{\wt\H_{p,p}})=(\wt{V_1}|_{\wt\H_{p,p}},\wt{V_2}|_{\wt\H_{p,p}})$ and it is a bishift-semigroup.
		
	Next we shall show that $V_1|_{\H_{p,u}}$ is a c.n.u semigroup and $V_2|_{\H_{p,u}}$ is a unitary semigroup on $\H_{p,u}.$ Note that $(\ob{V_2})^*|_{\hat\H_{p,u}}$ is a unitary semigroup on $\hat\H_{p,u}$ and  $(\ob{V_2})^*|_{\wt\H_{p,u}}=\wt{V_2}|_{\wt\H_{p,u}}$  is a unitary semigroup on $\wt\H_{p,u}.$ Hence $\ob{V_2}|_{\H_{p,u}}=V_2|_{\H_{p,u}}$ is a unitary semigroup on 	$\H_{p,u}.$	Note that 	
	\begin{equation}\label{Hpu}
		\hat\H_{p,u}=\spncl\{\ob{V_{1,t}}(\wt\H_{p,u}):t\in \R\}.
	\end{equation}
 Since  $\ob{V_1}^*$ is a unitary extension for the c.n.u semigroup $\wt{V_1}|_{\wt\H_{p,u}}$ on $\wt\H_{p,u},$ by \cref{Hpu}  one can see that $\hat\H_{p,u}$ is the minimal space for this extension. Therefore, since $\H_{p,u}=\hat\H_{p,u}\ominus\wt\H_{p,u},$  we see (using the same techniques as in \cref{dualc.n.u}) that $\ob{V_1}|_{\H_{p,u}}=V_1|_{\H_{p,u}}$ is also a c.n.u semigroup on $\H_{p,u}.$

 Similarly we can show that  $V_1|_{\H_{u,p}}$ is a  unitary semigroup, and $V_2|_{\H_{u,p}}$ is a c.n.u semigroup on $\H_{u,p}.$ This proves the theorem for a c.n.u pair $(V_1,V_2).$

  Now let  $(V_1,V_2)$ be any pair (not necessarily c.n.u) of dual doubly commuting semigroups of isometries. Let $\H_{u,u}$ be the unitary part in \cref{Cooper} for $V_1V_2.$ Then $\H_{u,u}$ reduces both  $V_1$ and $V_2,$ hence  $V_1|_{\H_{u,u}}$ and  $V_2|_{\H_{u,u}}$ are unitary semigroups.

 Let $\H_s=\H\ominus\H_{u,u}.$ Since the dual of $(V_1,V_2)$ is same as the dual of $(V_1|_{\H_s},V_2|_{\H_s}),$  $(V_1|_{\H_s},V_2|_{\H_s})$ is a c.n.u pair of dual doubly commuting semigroups of isometries. Hence by applying the above proof to the c.n.u pair $(V_1|_{\H_s},V_2|_{\H_s}),$ we get the proof of the theorem. 
\end{proof}

The converse of this theorem also holds and it is easy to see. Note that the dual of a commuting pair of a c.n.u semigroup and a unitary semigroup is again a commuting pair of a c.n.u semigroup and a unitary semigroup from  \cref{thm:cnuUmodel}.

The next result shows that $(\M_1,\M_2)$ of \cref{DDC-ex} is a model for the pairs of dual doubly commuting semigroups of isometries.
\begin{thm}\label{dual-model}
	Let $(V_1,V_2)$ be a modified-bishift-semigroup  on $\H.$ Then $(V_1,V_2)$ is jointly unitarily equivalent to $(\M_1\otimes I_\F,\M_2\otimes I_\F)$ for some Hilbert space $\F.$
\end{thm}

\begin{proof}
	Since the dual $(\wt{V_1},\wt{V_2})$ of $(V_1,V_2)$ is a bishift-semigroup, $(\wt{V_1},\wt{V_2})$ is jointly unitarily equivalent to $(\Sc_1\otimes I_\F,\Sc_2\otimes I_\F)$  for some Hilbert space $\F.$ 
	 Note that the dual of $(\Sc_1\otimes I_\F,\Sc_2\otimes I_\F)$ is $(\M_1\otimes I_\F,\M_2\otimes I_\F),$ hence by \cref{double-dual}, $(V_1,V_2)$  is jointly unitarily equivalent to $(\M_1\otimes I_\F,\M_2\otimes I_\F).$
\end{proof}

The forward part of the following corollary follows from \cref{dualdoublecommuting}, \cref{dual-model} and the fact that $(\M_1,\M_2)$ is not doubly commuting. The converse part follows from \cite{Fuglede} and \cref{dualdoublecommuting}.
\begin{corollary}
	Let $(V_1,V_2)$ be a  pair of commuting semigroups of isometries on $\H.$ Then $(V_1,V_2)$ is simultaneously doubly commuting and dual doubly commuting if and only if $\H$ decomposes as the direct sum $$\H=\H_{p,u}\oplus\H_{u,p}\oplus\H_{u,u}$$
	 of reducing subspaces $\H_{p,u},\H_{u,p}$ and $\H_{u,u}$ for both $V_1$ and $V_2,$	such that\\
	(1) $V_1$ is a c.n.u semigroup and $V_2$ is a unitary semigroup on $\H_{p,u},$\\
	(2) $V_1$ is a unitary semigroup and $V_2$ is a c.n.u semigroup on $\H_{u,p},$ and \\
	(3) both $V_1$ and $V_2$ are unitary semigroups on $\H_{u,u}.$
\end{corollary}

\noindent{\bf Acknowledgements}
 The authors thank the referee for a careful reading and useful suggestions. Research is funded by the J C Bose fellowship JCB/2021/000041, the  D S Kothari postdoctoral fellowship MA/20-21/0047 and the DST FIST program - 2021 [TPN - 700661].


\begin{thebibliography}{99}
 \bibitem{BCL}{Berger, C. A.; Coburn, L. A.; Lebow, A. Representation and index theory for $C\sp*$-algebras generated by commuting isometries. J. Functional Analysis 27 (1978), no. 1, 51--99. MR0467392}
	 \bibitem{Binzar-Lazureanu 2017}{Bînzar, Tudor; Lăzureanu, Cristian. On the Wold-type decompositions for $n$-tuples of commuting isometric semigroups. Filomat 31 (2017), no. 5, 1251--1264. MR3630043}
\bibitem{Binzar-Lazureanu 2018}{Bînzar, Tudor; Lăzureanu, Cristian. Wold-type decompositions for a commutative pair of noncommutative semigroups of isometries. Bull. Malays. Math. Sci. Soc. 41 (2018), no. 2, 1139--1150. MR3781563}
  \bibitem{BKPS} Burdak, Zbigniew; Kosiek, Marek; Pagacz, Patryk; Słociński, Marek. On the commuting isometries. Linear Algebra Appl. 516 (2017), 167--185. MR3589711
 \bibitem{Cooper}{Cooper, J. L. B. One-parameter semigroups of isometric operators in Hilbert space. Ann. of Math. (2) 48 (1947), 827--842. MR0027129}
 \bibitem{Fil}{Fillmore, Peter A. Notes on operator theory. Van Nostrand Reinhold Mathematical Studies, No. 30 Van Nostrand Reinhold Co., New York-London-Melbourne 1970 {\rm v}+122 pp. MR0257765}
 	\bibitem{Fuglede}{Fuglede, Bent. A commutativity theorem for normal operators. Proc. Nat. Acad. Sci. U.S.A. 36 (1950), 35--40. MR0032944}
	       \bibitem{GG}{Gaşpar, Dumitru; Gaşpar, Păstorel. Wold decompositions and the unitary model for bi-isometries. Integral Equations Operator Theory 49 (2004), no. 4, 419--433. MR2091468}


			\bibitem{Popovici II} Popovici, Dan. On the structure of c.n.u. bi-isometries. II. Acta Sci. Math. (Szeged) 68 (2002), no. 1-2, 329--347. MR1916584
	\bibitem{Popovici}{Popovici, Dan. A Wold-type decomposition for commuting isometric pairs. Proc. Amer. Math. Soc. 132 (2004), no. 8, 2303--2314. MR2052406}
	\bibitem{Ptak}{Ptak, Marek. Unitary dilations of multiparameter semigroups of operators. Ann. Polon. Math. 45 (1985), no. 3, 237--243. MR0817541}
		\bibitem{Rudin}{Rudin, Walter. Real and complex analysis. Third edition. McGraw-Hill Book Co., New York, 1987. {\rm xiv}+416 pp. ISBN: 0-07-054234-1 MR0924157}
\bibitem{Sarason}{Sarason, Donald. Generalized interpolation in $H\sp{\infty}$. Trans. Amer. Math. Soc. 127 (1967), 179--203. MR0208383}
  \bibitem{Slocinski-dilation}{Słociński, Marek. Unitary dilation of two-parameter semi-groups of contractions. Bull. Acad. Polon. Sci. Sér. Sci. Math. Astronom. Phys. 22 (1974), 1011--1014. MR0380509}
   \bibitem{Nagy-Foias}{Sz.-Nagy, Béla; Foias, Ciprian; Bercovici, Hari; Kérchy, László. Harmonic analysis of operators on Hilbert space. Second edition. Revised and enlarged edition. Universitext. Springer, New York, 2010. xiv+474 pp. ISBN: 978-1-4419-6093-1 MR2760647}
 		       \bibitem{Wold}{Wold, Herman. A study in the analysis of stationary time series. 2d ed. With an appendix by Peter Whittle. Almqvist and Wiksell, Stockholm, 1954. {\rm viii}+236 pp. MR0061344}
	\end{thebibliography}
\end{document}